\documentclass[eqno,12pt]{article}
\usepackage{amsmath,amsthm,amssymb,bm} 
\usepackage{datetime}

\newtheorem{theorem}{Theorem}
\newtheorem{lemma}[theorem]{Lemma}

\newtheorem{problem}[theorem]{Problem}

\begin{document}
\title{On a problem of Erd\H os and Moser}
\date{}
\author{
B\'ela Bollob\'as
\footnote{Department of Pure Mathematics and Mathematical Statistics, Wilberforce Road,
Cambridge CB3 0WB, UK; {\em and} Department of Mathematical Sciences, University of Memphis,
Memphis TN38152, USA; {\em and} London Institute for Mathematical Sciences, 35a South
St, Mayfair, London W1K 2XF, UK; email: bb12@cam.ac.uk. Research supported in part
by NSF grant ITR 0225610; and by MULTIPLEX no. 317532.
} 
\and
Alex Scott
\footnote{Mathematical Institute, University of Oxford, Andrew Wiles Building, Radcliffe Observatory
Quarter, Woodstock Road, Oxford OX2 6GG, UK; email: scott@maths.ox.ac.uk.
}
}
\maketitle

\begin{abstract}
A set $A$ of vertices in an $r$-uniform hypergraph $\mathcal H$ is {\em covered in $\mathcal H$} if there is some vertex $u\not\in A$ such that every edge of the form $\{u\}\cup B$, $B\in A^{(r-1)}$ is in $\mathcal H$.
Erd\H os and Moser (1970) determined the minimum number of edges in a graph on $n$ vertices such that every $k$-set is covered.  
We extend this result to $r$-uniform hypergraphs on sufficiently many vertices, and determine the extremal hypergraphs.
We also address the problem for directed graphs.
\end{abstract}

\section{Introduction}

Let $\mathcal H$ be an $r$-uniform hypergraph with vertex set $X$.  We say that a set $A\subset X$ is {\em covered in $\mathcal H$} if there is $u\not\in A$ such that, for every 
$B\in A^{(r-1)}$ we have $B\cup\{u\}\in\mathcal H$; we say that {\em $u$ covers $A$ (in $\mathcal H$)}.
Let $f(n,k,r)$ be the minimum number of edges in an $r$-uniform hypergraph $\mathcal H$ with vertex set $[n]=\{1,\dots,n\}$ 
such that every $k$-set in $[n]$ is covered in $\mathcal H$.   

The problem of determining $f(n,k,r)$ for graphs ($r=2$) was raised by Erd\H os and Moser in \cite{EM64}.  
Subsequently, they proved in \cite{EM70} that, for $r=2$ and all $n>k$,
$$ f(n,k,2)= (k-1)(n-1) - \binom{k-1}{2} + \left\lceil\frac{n-k+1}{2}\right\rceil,$$
with a unique extremal graph.
Our aim in this paper is to generalise this to larger values $r$: for fixed $k,r$ and sufficiently large $n$, we will determine the value of $f(n,k,r)$ and find the extremal graphs.

We begin with a construction giving an upper bound on $f(n,k,r)$.
Let 
$D(n,r)$
be the minimum number $|\mathcal H|$ of edges in an $r$-uniform hypergraph $\mathcal H$ with vertex set $[n]$ such that every $(r-1)$-set  is contained in some element of $\mathcal H$. 
It is clear that $D(n,r)\ge\binom{n}{r-1}/\binom{r}{r-1}$.  In fact, this bound is close to optimal: for fixed $r$, we have
$$D(n,r)=(1+o(1))\binom{n}{r-1}\bigg/\binom{r}{r-1}\sim n^{r-1}/r!.$$
This follows as a (very) special case of an important result of R\"odl \cite{R85}.
There is also a simple construction (pointed out to us by Noga Alon): 
for $c\in[n]$, consider the set $\{A\in[n]^{(r)}:\sum_{a\in A}a\equiv c\pmod n\}$.  This covers all $(r-1)$-sets except those $B$ for which $c-\sum_{b\in B}\in B \pmod n$. 
There are only $O(n^{r-2})$ of these, so we can cover them one at a time.

Now for $n>k\ge r-1$, let $\mathcal G(n,k,r)$ be the family of $r$-uniform hypergraphs $\mathcal G$ on vertex set $[n]$ that can be constructed as follows:
\begin{itemize}
\item add all edges that meet any of the vertices $\{1,\dots,k-r+1\}$;
\item on vertices $\{k-r+2,\dots,n\}$, we add a collection of $D(n-k+r-1,r)$ edges such that every $(r-1)$-set in $\{k-r+2,\dots,n\}$ is contained in at least one of these edges.
\end{itemize}
Let us check that for $\mathcal H\in \mathcal G(n,k,r)$, every $k$-set $B$ of vertices is covered in $\mathcal H$:
if $[k-r+1]\not\subseteq B$ then any element $i$ of $[k-r+1]\setminus B$ covers $B$.  Otherwise, 
$B\supseteq [k-r+1]$, so $B\setminus [k-r+1]$ has size $r-1$ and is therefore contained in some 
edge $C$ added in the second step of the construction: the unique element $u$ of $C\setminus B$ then covers $B$.
We note that our construction generalizes that of Erd\H os and Moser: for $r=2$, $\mathcal G(n,k,r)$ consists of a unique graph up to isomorphism.

Let $g(n,k,r)$ be the number of edges in each element of $\mathcal G(n,k,r)$, so
\begin{align}
g(n,k,r)
&=\binom{n}{r}-\binom{n-k+r-1}{r} + D(n-k+r-1,r)\label{gextremal}\\
&= (k-r+1+1/r+o(1))\frac{n^{r-1}}{(r-1)!}.\notag
\end{align}
Furthermore, if there exists a Steiner system with parameters $(n-k+r-1,r,r-1)$ (in other words, a perfect covering 
of $(r-1)$-sets by $r$-sets) then
\begin{equation}\label{gexact}
g(n,k,r)=\binom{n}{r}-\binom{n-k+r-1}{r} + \binom{n-k+r-1}{r-1}/r.
\end{equation}
By a recent major result of Keevash \cite{K14}, such designs exist for infinitely many values of $n-k+r+1$, and so 
\eqref{gexact} holds for infinitely many values of $n$.

The main result of this paper is to show that, for fixed $k,r$ and all sufficiently large $n$, the construction above is optimal (and that all hypergraphs of minimal size can be constructed in this way).

\begin{theorem}\label{maintheorem}
For every $k,r$ there is an integer $n_0(k,r)$ such that for all $n>n_0$, 
$$f(n,k,r)=g(n,k,r).$$  
Furthermore, if $\mathcal H$ is an $r$-uniform hypergraph with $n$ vertices and $f(n,k,r)$ edges 
in which every $k$-set of vertices is covered, then $\mathcal H$ is isomorphic to some element of $\mathcal G(n,k,r)$.
\end{theorem}

We note that
Theorem \ref{maintheorem} implies that, for $n>n_0(k,r)$,
$$f(n,k,r)\ge\binom{n}{r}-\binom{n-k+r-1}{r} + \binom{n-k+r-1}{r-1}/r,$$
with equality for infinitely many values of $n$.

We will prove Theorem \ref{maintheorem} in the next section; further discussion, and results for directed graphs, can be found in the final section.  We use standard notation throughout: in particular, $[n]=\{1,\dots,n\}$, $X^{(r)}=\{A\subset X:|A|=r\}$ and, for a hypergraph $\mathcal H$, we write $|\mathcal H|$ for the number of edges in $\mathcal H$.

\section{Covering hypergraphs}

The upper bound in Theorem \ref{maintheorem} follows immediately from the construction given in the introduction.
We will prove the lower bound by an induction on $r$, for which we will need two lemmas: the first lemma shows that, for fixed $k,r$ and large $n$, an extremal hypergraph must have $k-r+1$ vertices with degree close to the maximum possible;
the second will allow us to show that if an extremal hypergraph has $k-r+1$ vertices of almost maximum possible degree then it must in fact belong to
$\mathcal G(n,k,r)$.  We will state and prove both lemmas, and then complete the proof of Theorem \ref{maintheorem} at the end of the section.

Recall that if $\mathcal H$ is an $r$-uniform hypergraph then the {\em shadow} $\partial\mathcal H$ of $\mathcal H$ is the $(r-1)$-uniform hypergraph consisting of all $(r-1)$-sets that are contained in some edge of $\mathcal H$. 
We will need a version of the Kruskal-Katona Theorem (\cite{K63}, \cite{K68}; see also \cite{B86}, p.~30).  It will be convenient to use it in the simplified form 
due to Lov\'asz \cite{L79}: for $x\in[r,\infty)$, if $\mathcal H$ is an $r$-uniform hypergraph with $\binom xr$ edges, then
$\partial \mathcal H$ has at least $\binom{x}{r-1}$ edges.

We can now state the first lemma.  This will be used in our inductive argument, so if $r>2$ we will be able to assume 
that Theorem \ref{maintheorem} holds for $r-1$.

\begin{lemma}\label{lemmaa}
Let $k\ge r\ge 2$ be fixed, and if $r>2$ suppose that Theorem \ref{maintheorem} holds with $r-1$ in place of $r$
(we do not assume anything if $r=2$).
Let $\epsilon>0$ and let $n$ be sufficiently large (depending on $k,r,\epsilon$).  Let $\mathcal H$ be an $r$-uniform hypergraph with $n$ vertices and $f(n,k,r)$ edges, and suppose that every $k$-set of vertices is covered in $\mathcal H$.  Then $\mathcal H$ has $k-r+1$ vertices of degree at least $(1-\epsilon)\binom{n-1}{r-1}$.
\end{lemma}

\begin{proof}

Let $V$ and $E$ be the vertex and edge sets of $\mathcal H$ respectively, and write $m=|E|=f(n,k,r)$. 
Note that $m\le g(n,k,r)$, which by \eqref{gextremal} is at most
$(k-r+1+1/r+o(1))n^{r-1}/(r-1)!=O(n^{r-1})$.
Let $d_1\ge d_2\ge\dots\ge d_n$ be the degree sequence of $\mathcal H$, and suppose that $d_{k-r+1}<(1-\epsilon)\binom{n-1}{r-1}$.  Define $\eta$ by $(1-\eta)^{r-1}=1-\epsilon$.

Let $v$ be a vertex of minimal degree in $\mathcal H$, and let $\mathcal H_v$ be the neighbourhood hypergraph of $v$, that is  the $(r-1)$-uniform hypergraph with vertex set $V\setminus v$, and edge set $\{e\setminus v: \mbox{$v\in e$ and $e\in E$}\}$.
Since every $k$-set is covered in $\mathcal H$, it follows that if $r\ge 3$ then every $(k-1)$-set is covered in $\mathcal H_v$
(for any $(k-1)$-set $A\subset V\setminus v$, the set $A\cup\{v\}$ is covered in $\mathcal H$ by some $u$;
then $u$ covers $A$ in $\mathcal H_v$).
Thus $\mathcal H_v$ has at least $f(n-1,k-1,r-1)$ edges.
It follows by assumption that 
$f(n-1,k-1,r-1)=g(n-1,k-1,r-1)$ for sufficiently large $n$, and so
$\mathcal H$ has minimal degree
\begin{align}
\delta(\mathcal H)&\ge g(n-1,k-1,r-1)\notag\\
&\sim (k-r+1+1/(r-1)+o(1))\frac{n^{r-2}}{(r-2)!}.\label{deltabound}
\end{align}
If $r=2$, we have directly that $\delta(\mathcal H)\ge k$ or else any $k$-set containing $v$ and all its neighbours is not covered in $\mathcal H$.  So \eqref{deltabound} holds for all $r\ge 2$.

For $i=1,\dots,n$, define a real number $x_i\in[r-1,\infty)$ by
\begin{equation}\label{xdef}
d_i=\binom{x_i}{r-1},
\end{equation}
so $x_1\ge\dots \ge x_n$.
By the Kruskal-Katona Theorem, a vertex $v$ that covers $\binom{x}{r}$ $r$-sets must have degree at least $\binom{x}{r-1}$ (since $\mathcal H_v$ must contain all $(r-1)$sets in the shadow of the $r$-sets covered by $v$).  So the $i$th vertex in $\mathcal H$ covers at most $\binom{x_i}{r}$ $r$-sets.

Each $r$-set $R$ in $V$ is covered at least $k-r+1$ times in $\mathcal H$, or else we could choose a $k$-set $S$ containing $R$ and all vertices that cover $R$, and then $S$ would not be covered in $\mathcal H$.  Counting all pairs $(R,u)$ such that $R$ is an $r$-set and $u$ covers $R$, we see that
\begin{equation}\label{xbound}
\sum_{i=1}^n\binom{x_i}{r}\ge (k-r+1)\binom{n}{r}.
\end{equation}
We have $x_1\le n-1$ and, as
$d_{k-r+1}<(1-\epsilon)\binom{n-1}{r-1}$, we have  
from \eqref{xdef} and the definition of $\eta$ that
$$x_{k-r+1}\le(1-\eta+o(1))n.$$ 
Let $t=\lceil \log n\rceil^{r-1}$.  Then, as $\sum_{i=1}^nd_i=rm$, we have
$d_t\le rm/t=O(n^{r-1}/t)$
and so
$$x_t=O(n/t^{1/(r-1)})=O(n/\log n).$$
Furthermore, for every $i$ we have
$$\binom{x_i}{r}=\frac{x_i-r+1}{r}\binom{x_i}{r-1}=\frac{x_i-r+1}{r}d_i.$$
Since the sequence $(x_i)_{i=1}^n$ is decreasing, it follows that
\begin{align*}
\sum_{i=1}^{n}\binom{x_i}{r}
&= \sum_{i=1}^{n}\frac{x_i-r+1}{r}d_i\\
&\le \frac{x_1-r+1}{r}\sum_{i=1}^{k-r}d_i+\frac{x_{k-r+1}-r+1}{r}\sum_{i=k-r+1}^{t-1}d_i
   +\frac{x_t-r+1}{r}\sum_{i=t}^n d_i\\
&\le \frac{n-r}{r}\sum_{i=1}^{k-r}d_i+\frac{(1-\eta+o(1))n}{r}\sum_{i=k-r+1}^{t-1}d_i
   +O(nm/\log n)
\end{align*}
Now $m=O(n^{r-1})$, so the last term is $o(n^r)$.  Suppose 
$\sum_{i=1}^{k-r}d_i=\alpha\binom{n-1}{r-1}$ and $\sum_{i=k-r+1}^{t-1}d=\beta\binom{n-1}{r-1}$.
Then by \eqref{xbound} and the bound on $\sum_{i=1}^n\binom{x_i}{r}$ above we must have
\begin{align*}
(k-r+1)\binom nr 
&\le \frac{n-r}{r}\alpha\binom{n-1}{r-1}+\frac{(1-\eta+o(1))n}{r}\beta\binom{n-1}{r-1}
   +o(n^r)\\
&=\alpha\binom{n-1}{r}+(1-\eta+o(1))\beta\binom{n-1}{r}+o(n^r)\\
&=(\alpha+(1-\eta)\beta+o(1))\binom{n}{r},
\end{align*}
and so $\alpha+(1-\eta)\beta\ge k-r+1+o(1)$.  However, $\alpha\le k-r$ (as $d_1\le\binom{n-1}{r-1}$) and so
$\beta\ge1/(1-\eta)+o(1)$, giving
$$\alpha+\beta\ge k-r+\frac{1}{1-\eta}+o(1)\ge k-r+1+\eta,$$
for large enough $n$.  So
\begin{equation}\label{sum1}
\sum_{i=1}^{t-1}d_i \ge (k-r+1+\eta)\binom{n-1}{r-1}.
\end{equation}
On the other hand, by \eqref{deltabound},
\begin{align}
\sum_{i=t}^nd_i
&\ge(1+o(1))n\delta(\mathcal H)\notag\\
&\ge\big(k-r+1+1/(r-1)+o(1)\big)\frac{n^{r-1}}{(r-2)!}\notag\\
&=\big((r-1)(k-r+1)+1\big)\frac{n^{r-1}}{(r-1)!} + o(n^{r-1}).\label{sum2}
\end{align}
Since $\sum_{i=1}^n d_i=mr$, \eqref{sum1} and \eqref{sum2} imply that
$$m\ge (k-r+1+1/r+\eta/r+o(1))\binom{n-1}{r-1},$$
which gives a contradiction for large enough $n$.  We conclude that we must have $d_{k-r+1}\ge (1-\epsilon)\binom{n-1}{r-1}$ if $n$ is sufficiently large, which implies the result immediately.
\end{proof}

Our second lemma will allow us to clean up the structure of a hypergraph that is close to extremal.

\begin{lemma}\label{lemmab}
For every $r\ge2$ and $\lambda>0$ there is $\epsilon>0$ such that the following holds for all sufficiently large $n$.
Let $\mathcal F$ and $\mathcal H$ be two hypergraphs with vertex set $[n]$ such that  
\begin{itemize}
\item  $\mathcal F$ is $(r-1)$-uniform and $|\mathcal F|<\epsilon n^{r-1}$;
\item $\mathcal H$ is $r$-uniform;
\item  every $A\in[n]^{(r)}$ that contains an element of $\mathcal F$ is covered in $\mathcal H$.
\end{itemize}
Then there is an $r$-uniform hypergraph $\mathcal H'$ with vertex set $[n]$ such that
\begin{itemize}
\item $|\mathcal H'|+\lambda|\mathcal F|<|\mathcal H|$;
\item $\partial\mathcal H' \supset \partial\mathcal H$.
\end{itemize}
\end{lemma}

\begin{proof}
The proof hinges on the fact that the presence of $\mathcal F$ forces $\mathcal H$ to cover many $r$-sets, and this in turn
means that that $\mathcal H$ is not efficiently structured to have a large shadow (as some $(r-1)$ sets
are contained in multiple elements of $\mathcal H$).
We will obtain $\mathcal H'$ from $\mathcal H$ by deleting some edges, and possibly adding a smaller number of new edges to make sure that the shadow does not shrink.

We choose a large positive integer $M=M(r,\lambda)>0$ and then a small constant $\epsilon=\epsilon(M,r,\lambda)>0$.
Let $m=|\mathcal F|$. 
Let us say that an ordered $(r+1)$-tuple $(x_1,\dots,x_{r-1},w,u)$ is {\em good} if $\{x_1,\dots,x_{r-1}\}\in\mathcal F$ and $u$ covers $\{x_1,\dots,x_{r-1},w\}$ in $\mathcal H$. Since
there are $(r-1)!$ ways of ordering an element of $\mathcal F$, and $n-r+1$ choices for an additional vertex $w$ (with at least one choice of $u$ for each of these), there are
$(r-1)!(n-r+1)m$ sequences $(x_1,\dots,x_{r-1},w)$ that extend to a good $(r+1)$-tuple of form $(x_1,\dots,x_{r-1},w,u)$.

Fix a partition $X_1\cup \dots \cup X_{r+1}$ of $[n]$, and 
let $\mathcal B_0$ be the
set of good $(r+1)$-tuples $(a_1,\dots,a_{r+1})$ such that $a_i\in X_i$ for each $i$.
Let $\mathcal A_0$ be the set of $r$-tuples $(a_1,\dots,a_r)$ such that
$(a_1,\dots,a_{r+1})\in\mathcal B_0$ for some $a_{r+1}\in X_{r+1}$; and let $\mathcal F_0$
be the the set of $(r-1)$-tuples $(a_1,\dots,a_{r-1})$ such that $(a_1,\dots,a_r)\in \mathcal A_0$ for some $a_r\in X_r$. We will also describe this as follows: we consider $\mathcal B_0$
as a subset of the cartesian product $X_1\times\dots \times X_{r+1}$, and then $\mathcal F_0$ and $\mathcal A_0$ are the projections of $\mathcal B_0$ on to the first $r-1$ coordinates and
the first $r$ coordinates respectively.

We choose $X_1,\dots,X_{r+1}$ so that $|\mathcal A_0|$ is as large as possible: by 
considering a random partition, we see that
\begin{equation}\label{bound0}
|\mathcal A_0|\ge (r-1)!(n-r+1)m/(r+1)^{r+1}.
\end{equation}

For each $F=(a_1,\dots,a_{r-1})\in\mathcal F_0$, let $\alpha(F)$ be the number of distinct elements 
$u\in X_{r+1}$ such that $(a_1,\dots,a_{r-1},w,u)\in\mathcal B_0$ for some $w\in X_r$.
Note that $\alpha(F)\ge1$ for all $F\in\mathcal F_0$.

Suppose first that  
\begin{equation}\label{test1}
\sum_{F\in\mathcal F_0}\max\{\alpha(F)-1,0\}>\lambda m.
\end{equation}
In this case we construct our set system $\mathcal H'$ as follows.  For each 
$F=(a_1,\dots,a_{r-1})\in\mathcal F_0$ with $\alpha(F)>1$, 
we choose $\alpha(F)$ elements of $\mathcal B_0$ that 
have $F$ as an initial segment and contain distinct elements of $X_{r+1}$,
say
$$(a_1,\dots,a_{r-1},w_1,u_1),\dots,(a_1,\dots,a_{r-1},w_{\alpha(F)},u_{\alpha(F)})$$
(so the $u_i$ are distinct, but the $w_i$ may contain repetitions).
We now delete from $\mathcal H$ the edges
$$\{a_1,\dots,a_{r-1},u_1\},\dots,\{a_1,\dots,a_{r-1},u_{\alpha(F)-1}\};$$
these belong to $\mathcal H$ as $(a_1,\dots,a_r,w_i,u_i)$ is a good $(r+1)$-tuple for each $i$ and so $u_i$ covers $\{a_1,\dots,a_r,w_i\}$.
Repeating for each $F\in\mathcal F_0$ with $\alpha(F)>1$, it follows from \eqref{test1} that we delete a total of more than
$\lambda m$ edges from $\mathcal H$.  Let $\mathcal H'$ be the resulting hypergraph.  
Clearly $|\mathcal H'|+\lambda|\mathcal F|<|\mathcal H|$, so it is enough to show that $\partial\mathcal H'=\partial\mathcal H$.  Suppose that 
$I\in\partial\mathcal H$ (so $|I|=r-1$).  If $I\not\in\partial H'$ then there must be an edge of form $A=\{a_1,\dots,a_{r-1},u_i\}$ that contains $I$ and that we deleted from $\mathcal H$.  If $I=A\setminus a_j$ for some $j$, 
then $I\subset (A\setminus a_j)\cup\{w_i,u_i\}$, which belongs to $\mathcal H$ as $u_i$ covers
$\{a_1,\dots,a_r,w_i\}$
and it was not one of the edges we deleted; 
otherwise $I=\{a_1,\dots,a_{r-1}\}$, and is contained in $\{a_1,\dots,a_{r-1},u_{\alpha(F)}\}$, which we did not delete from $\mathcal H$.  We conclude that $\partial\mathcal H'=\partial\mathcal H$.

We may therefore assume that \eqref{test1} does not hold.  It follows that there are at most $\lambda m/M$ elements $F=(a_1,\dots,a_{r-1})\in\mathcal F_0$ with $\alpha(F)\ge M+1$. 
Let $\mathcal F_1=\{F\in\mathcal F_0: \alpha(F)\le M\}$, and let
$\mathcal A_1$ be the set of elements 
$(a_1,\dots,a_{r-1},w)\in\mathcal A_0$ 
such that $(a_1,\dots,a_{r-1})\in\mathcal F_1$.
Since each $F\in\mathcal F_0$ extends to at most $n-r+1$ elements of $\mathcal A_0$, it
follows from \eqref{bound0} that  for all sufficiently large $n$
$$|\mathcal A_1| \ge (r-1)!(n-r+1)m/(r+1)^{r+1} - (n-r+1)(\lambda m/M)\ge mn/M,$$
provided $M$ is large enough.   

For each $F=(a_1,\dots,a_{r-1})\in\mathcal F_1$, define
$$\beta(F)=\max_{u\in X_{r+1}}|\{w\in X_r: (a_1,\dots,a_{r-1},w,u)\in\mathcal B_0\}|,$$
and choose $u_F\in X_{r+1}$ such that there are $\beta(F)$ elements $w\in X_r$ such that
$(a_1,\dots,a_{r-1},w,u_F)\in\mathcal B_0$.   Each $F\in\mathcal F_1$ extends to at most $\alpha(F)\beta(F)$ elements of
$\mathcal B_0$, and so $F$ extends to at most $\alpha(F)\beta(F)$ elements of $\mathcal A_1$. Therefore
$$|\mathcal A_1|\le \sum_{F\in\mathcal F_1}\alpha(F)\beta(F)\le M\sum_{F\in\mathcal F_1}\beta(F),$$
and so $\sum_{F\in\mathcal F_1}\beta(F)\ge mn/M^2$.  

Let $\mathcal F_2=\{F\in\mathcal F_1:\beta(F)\ge n/2M^2\}$.  
So for each $F=(a_1,\dots,a_{r-1})\in\mathcal F_2$ we have chosen
$u_F\in X_{r+1}$ such that there are at least $n/2M^2$ elements of $\mathcal B_0$ of form $(a_1,\dots,a_{r-1},w,u_F)$ with $w\in X_r$.
By counting the number of ways of extending elements from $\mathcal F_1$ to $\mathcal A_1$, we see that
$$|\mathcal A_1|\le (n-r+1)|\mathcal F_2| + (n/2M^2)|\mathcal F_1\setminus\mathcal F_2|\le n|\mathcal F_2|+mn/2M^2.$$
Since $|\mathcal A_1|\ge mn/M$, this implies that $|\mathcal F_2|\ge m/2M^2$.

Now $\mathcal F_2$ is a subset of $X_1\times\dots\times X_{r-1}$.  For $i=1,\dots,r-1$, let $\mathcal P_i$ be the projection of $\mathcal F_2$ onto the coordinates other than $i$, 
so $\mathcal P_i\subset \prod_{j\in [r-1]\setminus i} X_j$.  It follows from the Loomis-Whitney Inequality \cite{LW49} (see also \cite{BT95}) that
$$\prod_{i=1}^{r-1}|\mathcal P_i| \ge |\mathcal F_2|^{r-2},$$
and so in particular there is some $i$ such that
$$|\mathcal P_i|\ge |\mathcal F_2|^{(r-2)/(r-1)} \ge (m/2M^2)^{(r-2)/(r-1)} >5 M^2 \lambda m/n,$$
as $m\le\epsilon n^{r-1}$ by assumption and we may also assume $\epsilon<1/ (10\lambda M^4)^r$.

Without loss of generality, we may assume that $i=r-1$.  Let $\mathcal P=\mathcal P_{r-1}$.  
For each element of $P=(a_1,\dots,a_{r-2})\in\mathcal P$,
choose an element $F\in\mathcal F_2$ that projects to $P$. Let $\mathcal F_3$ be the collection of $|\mathcal P|$ elements of $\mathcal F_2$ that are chosen.
Thus the elements of $\mathcal F_3$ have pairwise distinct projections on to the first $r-2$ coordinates.

We are finally ready to construct our modification of $\mathcal H$.  For $F=(a_1,\dots,a_{r-1})\in \mathcal F_3$,
let $W_F=\{w\in X_r:(a_1,\dots,a_{r-1},w,u_F)\in \mathcal B_0\}$, say $W_F=\{w_1,\dots,w_t\}$.  Since $F\in\mathcal F_3\subset\mathcal F_2$, we have $t\ge n/2M^2$;
if $t$ is odd, we reduce it by one (and do not use the final element of $W_F$).
Now delete from $\mathcal H$ the edges
\begin{equation}\label{deleting}
\{a_1,\dots,a_{r-2},w_1,u_F\},\dots,\{a_1,\dots,a_{r-2},w_t,u_F\}
\end{equation}
and add (if not already present) new edges
$$
\{a_1,\dots,a_{r-2},w_1,w_2\}, \{a_1,\dots,a_{r-2},w_3,w_4\}, \dots, \{a_1,\dots,a_{r-2},w_{t-1},w_t\}.
$$
Thus we decrease the number of edges by at least $t/2\ge n/5M^2$.  Repeating for every $F\in \mathcal F_3$, we decrease the number of edges by
at least  $|\mathcal F_3|(n/5M^2)> (5M^2\lambda m/n)(n/5M^2)=\lambda m$
(note that we delete disjoint sets of edges for distinct $F,F'\in\mathcal F_3$, as $F$ and $F'$ extend distinct elements of $\mathcal P$ and so have distinct projections on to the first $r-2$ coordinates).
Let $\mathcal H'$ be the resulting hypergraph.

All that remains is to check that $\partial H\subset \partial H'$.  Given $I\in\partial H\setminus \partial H'$, we know that $I$ must be contained in one of the edges we deleted.
So suppose without loss of generality that $I\subset \{a_1,\dots,a_{r-2},w_1,u_F\}$ in \eqref{deleting}.  Since 
$u_F$ covers $\{a_1,\dots,a_{r-1},w_1\}$ in $\mathcal H$  we cannot have $u_F\in I$, as then $I\cup \{a_{r-1}\}$ is an edge of both $\mathcal H$ and $\mathcal H'$.  But then $I\subset  \{a_1,\dots,a_{r-2},w_1,w_2\}\in\mathcal H'$.
This gives a contradiction and proves the lemma.
\end{proof}

We can now complete the proof of our main theorem.

\begin{proof}[Proof of Theorem \ref{maintheorem}]
We prove the theorem by induction on $r$.  The case $r=2$ is immediate from the Erd\H os-Moser theorem, so we may assume that $r\ge3$ and we have proved smaller cases.

We first note that, since Theorem \ref{maintheorem} holds for $r-1$, it follows that Lemma \ref{lemmaa} holds for $r$.  
Let $\mathcal H$ be a hypergraph with $n$ vertices and $f(n,k,r)$ edges in which every $k$-set is covered.
Let $v_1,\dots,v_n$ be the vertices of $\mathcal H$ in decreasing order of degree, 
let $V=\{v_1,\dots,v_n\}$, and let $A=\{v_1,\dots,v_{k-r+1}\}$.
Let $\mathcal H_1$ be  the restriction of $\mathcal H$ to $V\setminus A$.

Let $\mathcal G\subset V^{(r)}$ be the collection of all $r$-sets of vertices that meet $A$ but are not present in $\mathcal H$, and suppose $\mathcal G$ is nonempty.
It follows from Lemma \ref{lemmaa} that, for any $\epsilon>0$, if $n$ is sufficiently large then 
$|\mathcal G|<\epsilon n^{r-1}$.  

We split $\mathcal G$ into sets $\mathcal G_0,\dots,\mathcal G_{r-1}$, where 
$\mathcal G_i=\{G\in\mathcal G:|G\setminus A|=i\}$.
Let $\mathcal G_i$ be the largest of these sets, and let $\mathcal G^*_i=\{G\setminus A:G\in G_i\}$.
Let $\phi:(V\setminus A)^{(i)}\to (V\setminus A)^{(r-1)}$ be an injection such that $\phi(S)\supset S$ for every $S$ (the existence of such a mapping follows easily from Hall's Theorem) and let $\mathcal F=\{\phi(S):S\in\mathcal G_i^*\}$. 
Note that $|\mathcal G_i|\ge|\mathcal G|/r$ and for each $B\in\mathcal G_i^*$ there are at most $k^r$ sets $G\in\mathcal G_i$ with $G\setminus A=B$.
Thus $\mathcal F$ is a collection of 
$(r-1)$-sets in $V\setminus A$, such that $|\mathcal F|\ge|\mathcal G|/k^rr$ and every set in $\mathcal F$ 
contains $G\setminus A$ for some $G\in \mathcal G$.

Now for any $F\in \mathcal F$, we can pick $G\in\mathcal G$ such that $G\setminus A\subset F$ and choose $v\in A\cap G$.
For any $w\in V\setminus (F\cup A)$, the set $(A\setminus v) \cup F\cup \{w\}$ has 
size $k$ and so must be covered in $\mathcal H$ by some $u$.  We do not have $u=v$, as $G$ is not an edge of $\mathcal H$.  
Thus $u\not\in A$: it follows that 
$F\cup \{w\}$ is covered by $u$ in $\mathcal H_1$.

We can now apply Lemma \ref{lemmab} to the hypergraphs $\mathcal F$ and $\mathcal H_1$ (both on vertex set $V\setminus A$) with $\lambda=k^rr$ to deduce that there is some $\mathcal H_1'$ such that $|\mathcal H_1'|+k^rr|\mathcal F|<|\mathcal H_1|$ and
$\partial H_1'\supset\partial H_1$.  Let $\mathcal H'$ be the hypergraph with vertex set $V$ obtained from $\mathcal H_1'$
by adding all $r$-sets incident with $A$.  Then  $|\mathcal H'|<|\mathcal H|$, as $k^rr|\mathcal F|\ge|\mathcal G|$.  

Finally, we claim that every $k$-set in $V$ is covered in $\mathcal H'$.  Let $B\in V^{(k)}$.  
If $A\setminus B$ is nonempty, then any vertex of $A\setminus B$ covers $B$ in $\mathcal H'$.  
Otherwise, $A \subset B$: let $C=B\setminus A$, so $C$ is an $(r-1)$-set. 
There must be some $u\not\in B$ that covers $B$ in $\mathcal H$, and so $\{u\}\cup C\in\mathcal H$.  
Then we also have $\{u\}\cup C\in\mathcal H_1$, and so $C\in\partial \mathcal H_1$.  By construction, we also have $C\in\partial\mathcal H_1'$, so there is some 
$r$-set $D\in \mathcal H_1'$ that contains $C$.  Let $u'$ be the single element of $D\setminus C$: then $u'$ covers $B$ in $\mathcal H'$, as $D\in\mathcal H'$, and every other $r$-set contained in $\{u'\}\cup B$ meets $A$ and so belongs to $\mathcal H'$.
This gives a contradiction, as $|\mathcal H'|<|\mathcal H|$, and every $k$-set is covered in $\mathcal H'$.

We conclude that every edge incident with $A$ must be present in $\mathcal H$.  But now consider any $(r-1)$-set $B\subset V\setminus A$: arguing as above, the set $A\cup B$ must be covered by some $u$ in $\mathcal H$, and so $B\cup \{u\}\in\mathcal H_1$.  It follows that $\partial\mathcal H_1=(V\setminus A)^{(r-1)}$, and so 
$|\mathcal H_1|\ge D(n-|A|,r) = D(n-k+r-1,r)$.  Thus $f(n,k,r)=|\mathcal H|\ge g(n,k,r)$.
Furthermore, since $\mathcal H$ contains all $r$-sets meeting $A$, and a collection of
$D(n-k+r-1,r)$ $r$-sets in $V\setminus A$ containing all $(r-1)$-sets in $V\setminus A$ in their shadow,
we see that $\mathcal H\in\mathcal G(n,k,r)$.
\end{proof}

\section{Further directions}

There are a number of interesting related problems.  In their original paper \cite{EM64}, Erd\H os and Moser raised the question of covering oriented graphs.  A digraph with at least $k+1$ vertices has {\em property $S_k$} (named after Sch\"utte) if for every $k$-set $B$ of vertices there is a vertex $u$ such that $B$ is contained in the outneighbourhood of $u$.  Erd\H os and Moser \cite{EM64} asked the following.

\begin{problem}\label{orientedproblem}
What is the minimum number $F(n,k)$ of edges in a graph $G$ of order $n$ that has an orientation with property $S_k$?
\end{problem}

Erd\H os and Moser \cite{EM70} noted that, for fixed $k$ and large enough $n$, $f(k-1)n\le F(n,k) \le f(k)n$, where $f(k)$ is the minimum number of vertices in an oriented graph with property $S_k$; it is known that 
$f(k)=2^{(1+o(1))k}$.  

For $k\ge 2$, let $\delta(k)$ be the smallest positive integer $\delta$ such that some oriented graph with property $S_k$ has a vertex with indegree $\delta$.  We shall show that the asymptotic behaviour of $F(n,k)$ depends primarily on $\delta(k)$.

\begin{lemma}
For every $k\ge 2$ there is are integers $c=c(k)$ and $n_0=n_0(k)$ such that $F(n,k)=\delta(k)n+c$ for all $n\ge n_0$.
\end{lemma}

\begin{proof}
Note first that if $G$ is an oriented graph with $t$ vertices that has property $S_k$ and $v$ is any vertex of $G$ then we can generate another oriented graph
with property $S_k$ by adding a new vertex $v'$ with no outedges, and the same inneighbourhood as $v$.  If $G$ has a vertex with indegree $\delta(k)$ then applying
this construction repeatedly to a vertex of minimal indegree gives, for all $n\ge t$ 
and some $c_1=c_1(G)$,
oriented graphs with $n$ vertices, $\delta(k)n+c_1$ edges and property $S_k$.    

Any oriented graph $G$ with property $S_k$ must have at least $\delta(k)|V(G)|$ edges,
and so we must have $\delta(k)n\le F(k,n)\le\delta(k)n+c_1$ for all sufficiently large $n$.  Let $c\ge0$ be minimal such that there are infinitely many graphs $G$ with
$\delta(k)|V(G)|+c$ edges and property $S_k$: if $|V(G)|$ is large enough 
and $G$ has $\delta(k)|V(G)|+c$ edges
then it has minimal degree at most $\delta(k)$, so we can apply the construction to generate a sequence of oriented graphs
with $\delta(k)n+c$ edges and property $S_k$ for all sufficiently large $n$.
If $n$ is large enough, these are extremal.
\end{proof}

The covering problem is also natural for the more general class of digraphs (so we allow edges in both directions).  In this case, we can solve the problem completely.

Let $\mathcal A(n,k)$ be the family of digraphs on vertex set $[n]$ that can be constructed as follows:
\begin{itemize} 
\item Add all $k(k+1)$ directed edges that join any two elements from $[k+1]$
\item For each $i>k+1$, add edges from exactly $k$ elements of $[k+1]$ to $i$.
\end{itemize}
If $G \in\mathcal A(n,k)$ then every vertex of $G$ has indegree exactly $k$.  Let us check that every $k$-set $B$ is covered in $G$. If an element $i\in [k+1]$
does not cover $B$ then either $i\in B$, or $i\not \in B$ and some vertex of $B$ is not an outneighbour of  $i$.  There are $|B\cap[k+1]|$ elements of the first
type and at most $|B\setminus [k+1]|$ elements of the second type (as each element of $B\setminus[k+1]$ has edges from all but one element of $[k+1]$); so there is at least one element of $[k+1]$ left over to cover $B$.

\begin{lemma}
Let $k\ge2$, and let $D$ be a digraph of order $n$ that has property $S_k$.  Then $D$ has at least $kn$ edges, and if $D$ has exactly $kn$ edges then $D$ is isomorphic to some element of $\mathcal A(n,k)$.
\end{lemma}

\begin{proof}
Every vertex has indegree at least $k$, or else any $k$-set
that contains a vertex of minimal indegree and all its inneighbours is not covered.  Thus $D$ has at least $kn$ edges.

If $D$ has exactly $kn$ edges then every vertex has indegree exactly $k$.  Given a vertex $v$, let $\Gamma^-(v)$ be the set of inneighbours of $v$.  For any $w\in\Gamma^-(v)$,
the only vertex that can cover $\{v\}\cup \Gamma^-(v)\setminus w$ is $w$ (as it is the only inneighbour of $v$ outside the set).  
It follows that $\Gamma^-(v)$ induces a complete 
directed graph.
Now pick any $u\in\Gamma^-(v)$: we have $|\Gamma^-(u)\cap\Gamma^-(v)|=|\Gamma^-(v)\setminus u|=k-1$.  Let $A=\Gamma^-(u)\cup\Gamma^-(v)$, so $|A|=k+1$ and
every pair of elements from $A$ is joined in both directions
except possibly $u,u'$, where $u'$ is the unique element of $\Gamma^-(u)\setminus\Gamma^-(v)$.
Since $k\ge2$, we have $|A|\ge3$ and so we can pick $w\ne u,u'$ in $A$: then $\Gamma^-(w)=A\setminus w$ induces a complete directed graph, and so $u,u'$ are also joined in both directions.
Thus $A$ is complete and, as $|A|=k+1$, no further edges can enter $A$.

Finally, for any $x\not\in A$ and any $a,b\in A$, consider the $k$-set $\{x\}\cup A\setminus\{a,b\}$: this can only be covered by $a$ or $b$, and so one of $a$ and $b$ must send an edge to $x$.  Since this holds for every pair of elements in $A$, it follows that $k$ elements of $A$ must direct edges to $x$.  We conclude that $D\in\mathcal A(n,k)$.
\end{proof}

Erd\H os and Moser \cite{EM70} also considered a generalization of the graph problem, where every vertex must be covered by many other vertices.
Let $h(n,k,r,s)$ be the minimum number of edges in an $r$-uniform hypergraph $\mathcal H$ with vertex set $[n]$ such that every $k$-set in $[n]$ is covered by at least $s$ 
distinct vertices in $\mathcal H$ (so $h(n,k,r,1)=f(n,k,r)$).   
Erd\H os and Moser \cite{EM70} considered the case $r=2$ and noted that, for $n>n_0(k,s)$, 
\begin{equation}\label{sis2}
h(n,k,2,s)=f(n,k+s-1,2);
\end{equation}
they further noted that this equality does not hold for every $n,k,s$ (so the assumption that $n>n_0$ is necessary).
It seems likely that our methods should be useful for this problem, although we note that \eqref{sis2} does not hold when 2 is replaced by $r$
if $k$ is close to $r$: for instance 
with $k=r=3$ and $s=3$ the graph $\mathcal G(n,k+s-1,r)=\mathcal G(n,5,3)$ has $3$ vertices of large degree, but is not extremal as the edge containing these three vertices can be removed. 

In a similar vein, we can consider a multicolour version of the problem: let $b(n,k,r,c)$ be the minimum number of edges in 
a graph on $n$ vertices, in which the edges are partitioned into $c$ classes such that every $k$-set is covered by the edges in each class.  
How does $b(n,k,r,c)$ behave for fixed $k,r,c$ and large $n$?

\bigskip

\noindent{\bf Acknowledgement.}  The authors would like to thank Noga Alon for his helpful comments, and the referee for a careful reading.

\end{document}